\documentclass[reqno,final,11pt]{amsart}
\usepackage{natbib}  
\usepackage{fancyhdr} 
\usepackage{color} 
\usepackage{hyperref} 
\usepackage{graphicx} 
\usepackage{geometry}
\usepackage[utf8]{inputenc}
\usepackage[T1]{fontenc}
\usepackage{verbatim}

\usepackage{tikz}
\usetikzlibrary{calc,shapes,backgrounds,arrows}
\usepackage{amssymb}

\definecolor{aleacolor}{rgb}{0.16,0.59,0.78}
\hypersetup{
breaklinks,
colorlinks=true,
linkcolor=aleacolor,
urlcolor=aleacolor,
citecolor=aleacolor}

\renewcommand{\cite}{\citet}

\theoremstyle{plain}
\newtheorem{theorem}{Theorem}[section]

\newtheorem{lemma}[theorem]{Lemma}

\theoremstyle{definition}

\theoremstyle{remark}
\newtheorem{remark}[theorem]{Remark}

\makeatletter \@addtoreset{equation}{section} \makeatother
\renewcommand\theequation{\thesection.\arabic{equation}}
\renewcommand\thefigure{\thesection.\arabic{figure}}
\renewcommand\thetable{\thesection.\arabic{table}}

\newcommand{\N}{\mathbb{N}}

\newcommand{\R}{\mathbb{R}}

\newcommand{\calU}{\mathcal{U}}

\newcommand{\ind}[1]{\mathbbm{1}_{\left\{#1\right\}}}
\newcommand{\indset}[1]{\mathbbm{1}_{#1}}

\newcommand{\ceil}[1]{{\left\lceil #1 \right\rceil}}

\DeclareMathOperator{\E}{\mathbf{E}}
\renewcommand{\P}{\operatorname{\mathbf{P}}}

\newcommand{\crochet}[1]{{\langle #1 \rangle}}
\newcommand{\e}{\mathrm{e}}

\newcommand{\dd}{\mathrm{d}}

\renewcommand{\bar}[1]{\overline{#1}}
\renewcommand{\tilde}[1]{\widetilde{#1}}
\renewcommand{\hat}[1]{\widehat{#1}}
\renewcommand{\epsilon}{\varepsilon}
\renewcommand{\phi}{\varphi}

\usepackage{tcolorbox}
\tcbuselibrary{skins,breakable}

\newtcolorbox{pubnote}{
  colback=white,
  colframe=black!60,
  boxrule=0.5pt,
  arc=3pt,
  left=6pt,
  right=6pt,
  top=6pt,
  bottom=6pt,
  breakable
}

\usepackage{orcidlink}
\usepackage{lmodern} 
\usepackage{array}
\usepackage{subfig}
\usepackage{amssymb, amsfonts, amsmath,amsthm,bbm}
\usepackage{enumerate}

\usepackage{doi}

\begin{document} 

\title[Extremal Process of Last Progeny Modified Branching Random Walks]{Extremal Process of Last Progeny Modified\\ Branching Random Walks\footnote{
\begin{center}
\begin{pubnote}
\small
This is the Author Accepted Manuscript of the paper published in
\emph{ALEA  Latin American Journal of Probability and Mathematical Statistics}
23: Page 25--42 (2026).
The final published version differs in pagination.
\doi{10.30757/ALEA.v23-02}.
\end{pubnote}\end{center}
}}

\author[Partha Pratim Ghosh]{Partha Pratim Ghosh \orcidlink{0000-0002-4801-4538}}
\address{Institut für Mathematische Stochastik,
Technische Universität Braunschweig,
Universitätsplatz 2,
38106 Braunschweig,
Germany.}
\email{p.pratim.10.93@gmail.com}
\urladdr{\href{https://sites.google.com/view/parthapratim}{https://sites.google.com/view/parthapratim}}

\author[Bastien Mallein]{Bastien Mallein \orcidlink{0000-0002-9435-420X}}
\address{Institut de Mathématiques de Toulouse,
Université Toulouse III Paul Sabatier,
118 route de Narbonne,
31062 Toulouse Cedex 9, France.}
\email{bastien.mallein@math.univ-toulouse.fr}
\urladdr{\href{https://www.math.univ-toulouse.fr/~bmallein/}{https://www.math.univ-toulouse.fr/$\sim$bmallein}}


\begin{abstract}
We consider a last progeny modified branching random walk, in which the position of each particle at the last generation $n$ is modified by an i.i.d. copy of a random variable $Y$. Depending on the asymptotic properties of the tail of $Y$, we describe the asymptotic behaviour of the extremal process of this model as $n \to \infty$.
\end{abstract}

\maketitle

\section{Introduction}

Branching random walk (or BRW for short) is a particle system on the real line constructed as follows. It starts from a single particle at position $0$ forming the initial generation $0$ of the process. Each particle reproduces independently of all others by creating an identically distributed point process of children around its position. We denote by $\mathcal{U}$ the set of particles in the branching random walk. For all $u \in \mathcal{U}$, we write $S_u$ for the position of particle $u$, $|u|$ for the generation to which it belongs and $u_k$ for the ancestor of $u$ alive at generation $k \leq |u|$. We only consider in the present article supercritical branching random walks, satisfying the assumption
\begin{equation}
\label{eqn:supercritical}\E\left[ \#\{u \in \mathcal{U} : |u|=1\} \right] > 1.
\end{equation}
Under this condition, the BRW survives forever with positive probability,
or in other words \[\P(\#\{u \in \mathcal{U} : |u| = n\} > 0  \quad \forall\, n\geq0) > 0.\]
Note that we do not make any assumption on the finiteness of $\#\{u \in \mathcal{U} : |u|=1\}$, the random variable corresponding to the number of children of a given individual.

For all $\theta > 0$, we denote by $\kappa(\theta) = \log \E\left[ \sum_{|u|=1} \e^{\theta S_u} \right] \in (-\infty,\infty]$ the log-Laplace transform of (the intensity measure of the reproduction law of) the BRW. Assuming that $\kappa(\theta) < \infty$, we write
\[
  \kappa'(\theta) = \E\left[ \sum_{|u|=1} S_u \e^{\theta S_u-\kappa(\theta)} \right]
\]
whenever this integral is well-defined, irrespectively of the well-definition of $\kappa$ in a neighbourhood of $\theta$. If $\kappa'(\theta)$ is well-defined, we also write
\[
  \kappa''(\theta) = \E\left[ \sum_{|u|=1} (S_u-\kappa'(\theta))^2 \e^{\theta S_u - \kappa(\theta)} \right] \in [0,\infty].
\]

By Hölder's inequality, one immediately remarks that $\kappa$ is convex, therefore $\{\theta \in \R : \kappa(\theta) < \infty\}$ forms an interval. Moreover, by Lebesgue's dominated convergence theorem, it is straightforward to verify that $\kappa'(\theta)$ and $\kappa''(\theta)$ indeed correspond to the first and second derivative of $\kappa$ for  $\theta$ in the interior of that interval. We also remark that $\kappa''(\theta) = 0$ implies that almost surely, all particles in the first generation have the same position. We bar this degenerate situation from consideration by always assuming throughout this article that
\begin{equation}
  \label{eqn:non-degenerate}
  \forall x \in \R, \quad  \P(\exists u\in\calU: |u|=1,  S_u \neq x) > 0.
\end{equation}

Branching random walks have been the subject of a large and still expanding literature. One of the most studied features of this model is the asymptotic behaviour of the position of particles at the right tip of the BRW. \cite{Big2} proved that the \emph{maximal displacement} $M_n := \max_{|u|=n} S_u$ satisfies
\[
  \lim_{n \to \infty} \frac{M_n}{n} = v := \inf_{\theta > 0} \frac{\kappa(\theta)}{\theta} \quad \text{a.s. on the survival event of the BRW,}
\]
as long as there exists $\theta > 0$ such that $\kappa(\theta)< \infty$.
\cite{ABR09}  (see also \citealp{BraZ}) showed that if there exists $\theta_0 > 0$ such that
\begin{equation}
\label{eqn:defTheta0}
  \theta_0\kappa'(\theta_0) -\kappa(\theta_0) = 0,
\end{equation}
and additional integrability conditions hold, then setting $m_n := nv - \frac{3}{2\theta_0} \log n$, the sequence $(M_n - m_n)$ is tight, although \cite{HuS09} proved that this sequence exhibits almost sure fluctuations on a logarithmic scale. \cite{Aid} proved the convergence in distribution of the centred maximal displacement under close to optimal integrability conditions for a BRW satisfying~\eqref{eqn:defTheta0}.

In order to study the joint convergence in distribution of the first few rightmost particles, one can consider the so-called \emph{extremal process} of the BRW, defined as
\[
  \mathcal{Z}_n := \tau_{-m_n} \mathcal{X}_n,
\]
where $\mathcal{X}_n = \sum_{|u|=n} \delta_{S_u}$ is the empirical measure of the BRW and $\tau_x$ is the operator corresponding to a shift by $x$ of the point measure. \cite{Mad17} proved that under mild conditions, the extremal process $\mathcal{Z}_n$ of the BRW converges in law for the topology of vague convergence of point measure to a limiting decorated Cox process $\mathcal{Z}_\infty$. We refer to~\cite{Shi} for an overview of branching random walks, and to \cite{MaM} and \cite{SuZ15} for background on decorated Cox processes and their connections with branching particle systems.

In this article, we are interested in the \emph{last progeny modified BRW}, introduced by \cite{BG23}, which can be constructed as follows: Let $\nu$ be a probability measure on $\R$ and $(Y_u, u \in \mathcal{U})$ be a collection of i.i.d. random variables of law $\nu$, which are independent of the BRW. This model is the family of point measures defined for $n \geq 0$ by
\[
  \mathcal{E}_n = \sum_{|u|=n} \delta_{S_u + Y_u}.
\]
In other words, in the process $\mathcal{E}_n$, the positions of particles at the last step $n$ in the BRW are modified by the i.i.d. random variables $(Y_{u},|u|=n)$. This can be seen as a noisy observation of the positions of particles in the branching random walk.

\cite{BG23} considered a last progeny modified BRW with perturbation law $\nu$ given by the Gumbel law of parameter $1/\theta$. Assuming that $\theta \kappa'(\theta) - \kappa(\theta) \leq 0$, they showed that
\begin{equation}
  \label{eqn:cvExtremalProcess}
  \lim_{n \to \infty} \tau_{- m_n(\theta)} \mathcal{E}_n = \mathcal{Y}_\theta \quad \text{ in law for the topology of vague convergence},
\end{equation}
with $\mathcal{Y}_\theta$ a random point measure. On the one hand, if $\theta \kappa'(\theta) - \kappa(\theta) < 0$, then $ m_n(\theta) = n \frac{\kappa(\theta)}{\theta}$ and the limiting point measure $\mathcal{Y}_\theta$ is a Cox process (i.e. a Poisson point process with a random intensity measure) with intensity $\theta W_\infty(\theta) \e^{-\theta x} \,\dd x$. On the other hand, if $\theta \kappa'(\theta) - \kappa(\theta) = 0$,~\eqref{eqn:cvExtremalProcess} holds with $ m_n(\theta) = n \kappa'(\theta) - \frac{1}{2\theta} \log n$ and $\mathcal{Y}_\theta$ a Cox process with intensity  $\sqrt{\frac{2}{\pi \kappa''(\theta)}}\theta Z_\infty \e^{-\theta x} \,\dd x$.
 Here, $W_\infty(\theta)$ and  $Z_\infty$ are the almost sure limits of the \emph{additive martingale} $(W_n(\theta),n\geq 0)$ and the \emph{derivative martingale} $(Z_n,n\geq0)$, as defined in~\eqref{Def:Wn} and~\eqref{Def:Zn}, respectively.
 These results were extended by  \cite{Kow23}, who studied the asymptotic behaviour of the position of the right-most atom in the last progeny modified BRW for $\nu$ belonging to a larger class of probability distributions.

The main objective of the present article is to study the asymptotic behaviour of the full extremal process $\mathcal{E}_n$ as $n \to \infty$, assuming that the tail of the law $\nu$ is exponential, more precisely that there exists $\theta > 0$ and a regularly varying function $L$ such that
\begin{equation}
  \label{eqn:rvNu}
  \nu([x,\infty)) \sim L(x) \e^{-\theta x} \quad \text{ as $x \to \infty$.}
\end{equation}
Let us recall that a function $L$ is called regularly varying at $\infty$ with parameter $\alpha \in \R$ if for all~$\lambda > 0$, we have
\[
  \lim_{x \to \infty} \frac{L(\lambda x)}{L(x)} = \lambda^\alpha.
\]
A function is called \emph{regularly varying at $\infty$} if there exists such a parameter $\alpha$. We refer to the book by  \cite{BGT} for background on regularly varying functions.

We endow the set of point measures on $\R$ with the topology of vague convergence. For $\mu$ a Radon measure on $\R$ and $Q$ a non-negative random variable, we denote by PPP($Q \mu$) a Cox process with intensity $Q \mu$ on $\R$, i.e., conditionally on $Q$, a Poisson point process with intensity $Q \mu$.

We show in this article that the asymptotic behaviour of $\mathcal{E}_n$ is mainly predicted by the parameter~$\theta$ in~\eqref{eqn:rvNu}. More precisely, it satisfies a different asymptotic depending on whether $\theta < \theta_0$, $\theta = \theta_0$ or $\theta > \theta_0$ with $\theta_0$ the critical parameter of the BRW defined in~\eqref{eqn:defTheta0}. We refer to these three asymptotic regimes as subcritical, critical and supercritical, respectively, and we present our main results in each of these three regimes in the next three sections. We end the introduction with an heuristic explanation for our results.

\begin{remark}
Aiming for generality of our results, we do not wish to work under the assumption that the parameter $\theta_0$ exists except if necessary. However, by strict convexity of $\kappa$, we remark that $\theta \mapsto \theta \kappa'(\theta) - \kappa(\theta)$ is increasing whenever it is well-defined. Therefore, we will determine if the last progeny modified BRW is in the subcritical, critical or supercritical regime by observing the sign of $\theta \kappa'(\theta) - \kappa(\theta)$.
\end{remark}

\subsection{Extremal process in the subcritical regime}

Let $\nu$ be a probability distribution satisfying~\eqref{eqn:rvNu} and $S$ a BRW such that $\kappa(\theta) < \infty$. In the subcritical regime, we claim that the asymptotic behaviour of $\mathcal{E}_n$ can be described using the so-called \emph{additive martingale} of the BRW, defined as follows:
\begin{align}
\label{Def:Wn}
  W_n(\theta) := \sum_{|u|=n} \e^{\theta S_u - n \kappa(\theta)}.
\end{align}
It is immediate from its definition that $(W_n(\theta), n \geq 0)$ is a non-negative martingale, therefore converges almost surely to a limit $W_\infty(\theta)$. Assuming that $\kappa'(\theta)$ is well-defined,  \cite{Big} obtained a necessary and sufficient condition for its uniform integrability (with an alternative proof by \citealp{Lyo95} based on a spine decomposition argument). More precisely, if $\kappa'(\theta)$ is well-defined,
then the martingale $(W_n(\theta), n\geq0)$ is uniformly integrable if and only if
\begin{equation}
  \label{eqn:ncsNormal}
  \theta \kappa'(\theta) - \kappa(\theta) < 0 \text{ and } \E\left[W_1(\theta) \log_+ W_1(\theta)\right] < \infty,
\end{equation}
where $\log_+(x) = \log \max(x,1)$. If~\eqref{eqn:ncsNormal} does not hold, then $W_\infty(\theta) = 0$ a.s. This result was extended by \cite{AlI09}, who obtained a necessary and sufficient condition for the uniform integrability of $(W_n(\theta), n\geq0)$ regardless of the well-definition of $\kappa'(\theta)$.

For the last progeny modified BRW, we show that the asymptotic behaviour of its extremal process is in the subcritical regime whenever $(W_n(\theta), n \geq 0)$ is uniformly integrable, i.e., as soon as the conditions of Alsmeyer and Iksanov are satisfied. However, the more restrictive condition~\eqref{eqn:ncsNormal} allows us to work with a larger class of perturbation distribution $\nu$. We therefore state two versions of our results in the subcritical regime, the first one holding under minimal conditions on the BRW but with stronger assumptions on $\nu$, the second under minimal conditions for the law $\nu$ but with the more restrictive conditions on the BRW.

We first consider a BRW satisfying the assumptions of~\cite[Theorem~1.3]{AlI09} and a perturbation distribution $\nu$ having an exact exponential tail.
\begin{theorem}
\label{thm:miniBrw}
Let $\theta > 0$ such that $\kappa(\theta) < \infty$. We assume that $(W_n(\theta), n \geq 0)$ is uniformly integrable. Let $\nu$ be a probability distribution on $\R$ such that there exists a constant $L\in (0,\infty)$ satisfying
\begin{equation}
  \label{eqn:expTail}
  \nu([x,\infty)) \sim L \e^{-\theta x}\quad \text{as $x \to \infty$.}
\end{equation}
Then, writing
\[
 m_n = n \frac{\kappa(\theta)}{\theta} + \frac{1}{\theta} \log L,
 \]
the extremal process $\tau_{-m_n} \mathcal{E}_n$ converges in law to a PPP$(\theta W_\infty(\theta)\e^{-\theta x} \,\dd x)$.
\end{theorem}

Next, we consider a BRW satisfying~\eqref{eqn:ncsNormal} and a perturbation distribution $\nu$ satisfying~\eqref{eqn:rvNu}.
\begin{theorem}
\label{thm:miniLaw}
Let $\theta>0$ such that $\kappa(\theta) < \infty$. We assume that there exists $\delta > 0$ such that $\kappa(\theta+\delta) + \kappa(\theta - \delta) < \infty$ and that~\eqref{eqn:ncsNormal} holds. Let $\nu$ be a probability distribution on $\R$ satisfying~\eqref{eqn:rvNu} with a regularly varying function $L$ of index $\alpha$. Then, writing
\[
  m_n = n \frac{\kappa(\theta)}{\theta} + \frac{1}{\theta} \log L(n) \quad \text{and} \quad c_1 = \left(\frac{\kappa(\theta)}{\theta} - \kappa'(\theta) \right)^\alpha,
\]
the extremal process $\tau_{-m_n} \mathcal{E}_n$ converges in law to a PPP$(c_1\theta W_{\infty}(\theta) \e^{-\theta x} \,\dd x)$.
\end{theorem}

Noting that a constant is a regularly varying function of index $0$, we see that these two statements are consistent. Theorem~\ref{thm:miniBrw} corresponds to the case $\alpha=0$ and $c_1=1$; however, for its conclusion to hold, it is not necessary to assume that $\kappa$ is finite at any point other than $\theta$, or that $\kappa'(\theta) > -\infty$, unlike in Theorem~\ref{thm:miniLaw}.
 The conditions we worked under in these results were chosen to obtain an explicit centring sequence $(m_n)$. However, we believe that it is a generic result in the subcritical regime that there will exist a centring sequence such that $\tau_{-m_n} \mathcal{E}_n$ converges in law to a Poisson point process with intensity $W_\infty(\theta) \e^{-\theta x} \dd x$.

\subsection{Extremal process in the critical regime}

Let $S$ be a BRW, we assume that there exists $\theta_0 > 0$ such that~\eqref{eqn:defTheta0} holds, and that $\nu$ satisfies~\eqref{eqn:rvNu} with $\theta = \theta_0$. We refer to this situation as the critical regime. We observe that the martingale $(W_n(\theta_0), n\geq0)$ converges to 0 almost surely, by the previously mentioned result of Biggins. However, the so-called \emph{derivative martingale} defined for $n \geq 0$ by
\begin{align}
\label{Def:Zn}
  Z_n := \sum_{|u|=n} (n \kappa'(\theta_0) - S_u) \e^{\theta_0 S_u - n \kappa(\theta_0)}
\end{align}
works as a good replacement for this martingale. Immediate computations show that $(Z_n, n\geq 0)$ is a signed martingale, which is generally not uniformly integrable. Assuming that $\kappa''(\theta)  <\infty$, \cite{Aid} obtained necessary condition that \cite{Che15} proved to be sufficient for the convergence of $Z_n$ to an a.s. non-negative non-degenerate limit $Z_\infty$. This necessary and sufficient condition is
\begin{equation}
  \label{eqn:ncsBoundary}
  \E\left[ W_1(\theta_0) (\log_+ W_1(\theta_0))^2 \right] + \E\left[\bar{W}_1 \log_+(\bar{W}_1)\right] < \infty,
\end{equation}
where $\bar{W}_1 = \sum_{|u|=1} (\kappa'(\theta_0) - S_u)_+ \e^{\theta_0 S_u - \kappa(\theta_0)}$ and $x_+ = \max(x,0)$.

Under these conditions, as well as a slightly more stringent condition on the tail of $\nu$, we obtain a similar result to the subcritical regime, with $Z_\infty$ playing the role of $W_\infty(\theta)$, and a modified centring term for the extremal process.
\begin{theorem}
\label{thm:derBrw}
Let $\theta > 0$ such that $\kappa(\theta) < \infty$ and $\kappa''(\theta) < \infty$. We assume that
\[ \theta \kappa'(\theta) - \kappa(\theta) =0,\]
i.e., $\theta = \theta_0$ and that~\eqref{eqn:ncsBoundary} holds. Let $\nu$ be a probability distribution on $\R$ such that there exists a regularly varying function $L$ at $\infty$ with index $\alpha\in(-2,0)$ satisfying~\eqref{eqn:rvNu}. Then, writing
\begin{align*}
  m_n = n \frac{\kappa(\theta)}{\theta} + \frac{1}{\theta} \log L\left(\sqrt{n}\right) - \frac{1}{2\theta} \log n  \quad \text{and} \quad c_2 =\sqrt{\frac{2}{\pi \kappa''(\theta)}}\left( 2 \kappa''(\theta) \right)^{\frac{\alpha}{2}} \Gamma\left(\frac{\alpha}{2}+1\right),
\end{align*}
the extremal process $\tau_{-m_n} \mathcal{E}_n$ converges in law to a PPP$(c_2 \theta Z_\infty \e^{-\theta x} \,\dd x)$.
\end{theorem}

\begin{remark}
Theorem~\ref{thm:derBrw} also holds assuming that $\nu$ satisfies~\eqref{eqn:rvNu} with $L$ a positive constant, using the same proof techniques as the ones we develop below. This result therefore extends the previous estimate of \cite{BG23} for $\nu$ a Gumbel distribution with parameter~$\frac{1}{\theta_0}$.
\end{remark}

\begin{remark}
We believe that up to adding a stronger integrability condition on the reproduction law of the BRW, one could prove a result similar to Theorem~\ref{thm:derBrw} assuming that $\nu$ satisfies~\eqref{eqn:rvNu} with $L$ a regularly varying function with parameter $\alpha \geq 0$. However, we predict a sharp phase correction around $\alpha = -2$, and that the limiting extremal process should resemble the one obtained in the supercritical regime if $\alpha < -2$.
\end{remark}

\subsection{Extremal process in the supercritical regime}

We finally turn to the case of perturbation law $\nu$ satisfying~\eqref{eqn:rvNu} with $\theta > \theta_0$. In this situation, the perturbation has such strong tails that its asymptotic behaviour can be deduced from the asymptotic behaviour of $\mathcal{Z}_n$, the extremal process of the BRW. As a result, we consider here a BRW satisfying the assumptions used by~\mbox{\cite{Mad17}} to prove the convergence of $\mathcal{Z}_n$. These conditions include notably~\eqref{eqn:ncsBoundary}. In this situation, it will be sufficient for the tail of $\nu$ to be light enough for the convergence of $\mathcal{E}_n$ to hold.

\begin{theorem}
\label{thm:aboveBoundary}
We assume that the reproduction law of the BRW is non-lattice, and there exists $\theta_0 > 0$ satisfying~\eqref{eqn:defTheta0}. We assume that $\kappa''(\theta_0) < \infty$ and that~\eqref{eqn:ncsBoundary} holds. Let $\nu$ be a probability measure such that there exist $C > 0$ and $\theta > \theta_0$ verifying
\begin{equation}
 \label{eqn:NuBounded}
   \nu([x,\infty)) \leq C \e^{-\theta x} \text{ for $x \in \R$.}
\end{equation}
Then, writing
\[
m_n = n \kappa'(\theta_0) - \frac{3}{2\theta_0} \log n,
\]
 the extremal process $\tau_{-m_n} \mathcal{E}_n$ converges in law to
\begin{equation}
  \label{eqn:aboveBoundary}
  \sum_{i \in \N} \delta_{z_i + Y_i},
\end{equation}
where $(z_i, i \in \N)$ are the atoms of the limiting extremal process $\mathcal{Z}_\infty$ of the BRW and $(Y_i)$ is an independent sequence of i.i.d. variables with law $\nu$.
\end{theorem}

\subsection{Heuristics for our results}

The asymptotic behaviour of $\mathcal{E}_n$ can be explained by observing the distribution of particles at the $n$th generation of the BRW. We assume here that $\kappa(\theta)$ is finite for all $\theta \geq 0$, and that there exists $\theta_0$ satisfying~\eqref{eqn:defTheta0}. In this situation, for all $\theta < \theta_0$ it is well-known that the contributions to the value of the martingale $W_n(\theta)$ are mostly supported by particles in a neighbourhood of $n \kappa'(\theta)$. Otherwise stated, there are at time $n$ approximately $W_\infty(\theta)\e^{-n (\theta \kappa'(\theta) - \kappa(\theta))}$ particles to the neighbourhood of $n \kappa'(\theta)$. Moreover, with high probability, there is no particle to the right of $n \kappa'(\theta_0)$.

Let us now estimate the maximal position reached by $S_u + Y_u$ by a particle $u$ such that $S_u$ is close to $n \kappa'(\theta)$. Under the assumption that $\nu$ has an exponential tail with parameter $\vartheta$, we know that the maximum of $N$ independent copies of $Y$ will be close to $\frac{\log N}{\vartheta}$. Therefore, we expect particles in a neighbourhood of $n \kappa'(\theta)$ to create a cloud of perturbed observations with the furthest point being found around
\[
  n \left[ \kappa'(\theta) + \frac{1}{\vartheta} \left( \kappa(\theta)  - \theta \kappa'(\theta) \right) \right] + \frac{1}{\vartheta}\log W_\infty(\theta).
\]
Differentiating the linearly growing term with respect to $\theta$, we obtain $\kappa''(\theta)(1 - \theta/\vartheta)$, therefore this quantity will be maximal, for large $n$, when $\theta = \min(\vartheta,\theta_0)$.

When $\vartheta < \theta_0$, the extremal process $\mathcal{E}_n$ is thus obtained by generating a large number of i.i.d. copies of $Y$. As a result, this extremal process, centred by $n \kappa(\vartheta)/\vartheta$ will converge to a Poisson point process with intensity $\vartheta\e^{-\vartheta x}\, \dd x$, shifted by $\frac{1}{\vartheta}\log W_\infty(\vartheta)$ to take into account the exact number of random variables. This corresponds to the situation described in Theorems~\ref{thm:miniBrw} and \ref{thm:miniLaw}.

When $\vartheta \geq \theta_0$, then the maximum of the last progeny modified BRW is generated by one of the leading particles, in a neighbourhood of $n \kappa'(\theta_0)$. This explains the shape of the limiting extremal process in the supercritical regime, which is obtained by perturbing the limiting extremal process of the BRW. In Theorem~\ref{thm:aboveBoundary}, the tail of the perturbations is so thin that only particles very close to the edge create a large perturbation, which yields the natural form of the limiting extremal process, obtained by perturbing  the limiting extremal process of the original BRW with i.i.d. random variables. In Theorem~\ref{thm:derBrw}, the main contribution near the tip of $\mathcal{E}_n$ comes from particles at distance~$\Theta(\sqrt{n})$ from the leading edge of the BRW, which are counted   using the derivative martingale of the BRW. This explains the extremal process, given by a Poisson point process with exponential intensity, shifted by $\frac{1}{\theta_0}\log Z_\infty$ instead of $\frac{1}{\theta_0}\log W_\infty(\theta_0)$, and the apparition of $L(\sqrt{n})$ in the centring term $m_n$, as the typical perturbations observed at generation $n$ make jumps of that size $z \in \Theta(\sqrt{n})$, with probability approximately~$L(z) \e^{-z}$.

Our proof techniques consist in validating this heuristic computation, by identifying the set of particles in the BRW that contribute to the generation of the tip of the point measure $\mathcal{E}_n$, then carefully proving the convergence towards the identified extremal process.

\subsection{Organization of the article.}
We introduce in the next section some estimates that allow us to study the Laplace transform of $\mathcal{E}_n$. We then use the convergence of this Laplace transform to prove our main theorems in Section~\ref{sec:proofs}.

\section{Laplace transform of the last progeny modified branching random walk}

To prove the convergence in distribution of $\tau_{-m_n} \mathcal{E}_n$ to a limiting point measure $\mathcal{Z}$ in law for the topology of vague convergence, it is sufficient to show that for any continuous compactly supported function $\phi$, we have
\[
  \lim_{n \to \infty} \crochet{\tau_{-m_n} \mathcal{E}_n,\phi} = \crochet{\mathcal{Z},\phi} \quad \text{ in law,}
\]
where we write $\crochet{\mathcal{X},\phi} = \int \phi \,\dd \mathcal{X}$. As a result, it is sufficient to prove the convergence of the Laplace transform of $\tau_{-m_n}\mathcal{E}_n$, defined as
\[
  \phi \in \mathcal{T} \mapsto \E\left[ \exp\left( - \crochet{\tau_{-m_n}\mathcal{E}_n,\phi} \right) \right],
\]
where $\mathcal{T}$ is the set of all non-negative continuous compactly supported functions, i.e., to show that for each $\phi \in \mathcal{T}$, we have
\begin{equation}
  \label{eqn:laplaceConvergence}
  \lim_{n \to \infty}\E\left[ \exp\left( - \crochet{\tau_{-m_n} \mathcal{E}_n,\phi} \right) \right] = \E\left[ \exp\left( - \crochet{\mathcal{Z},\phi} \right) \right].
\end{equation}
In the rest of the section, we introduce some methods and results allowing to study this Laplace transform.

Let $x \in \R$ and $n \in \N$. Using the independence between $(Y_u, u \in \mathcal{U})$ and the BRW, we observe that
\begin{equation}
  \label{eqn:lapTransform}
    \E\left[ \exp\left( - \crochet{\tau_x \mathcal{E}_n,\phi} \right) \right]
  =  \E\left[ \exp\left( -\sum_{|u|=n} g_\phi(x + S_u)\right) \right] =  \E\left[ \exp\left( - \crochet{\tau_x \mathcal{X}_n,g_\phi} \right) \right]
\end{equation}
where we have set
\[
  g_\phi : x \mapsto -\log \int \e^{-\phi(x + y)}\, \nu(\dd y),
\]
and we recall that $\mathcal{X}_n$ is the counting measure of the BRW at time $n$.

Applying Madaule's theorem (\citealp{Mad17}) to~\eqref{eqn:lapTransform}  yields Theorem~\ref{thm:aboveBoundary}  (see Section~\ref{subsec:above}).
 To prove Theorems~\ref{thm:miniBrw},~\ref{thm:miniLaw}, and~\ref{thm:derBrw}, we first show that $M_n - m_n$ converges to~$-\infty$ in probability, where we recall that $M_n := \max_{|u|=n} S_u$ (with $m_n$ as defined in each theorem), and then study the asymptotic behaviour of $g_\phi(x)$ as $x \to -\infty$.
 Let us first state the following generic lemma for BRW, which follows from \cite[Lemma 3.1]{Shi}, using the transformation $x \mapsto -\theta x + \kappa(\theta)$ on the underlying point process.
\begin{lemma}
\label{lem:asCv}
If $\kappa(\theta) < \infty$, then $\lim_{n \to \infty} \theta M_n - n \kappa(\theta) = -\infty$ a.s.
\end{lemma}

Lemma~\ref{lem:asCv} is enough to conclude that $ M_n - m_n \to -\infty$ under the assumptions of Theorems~\ref{thm:miniBrw} and~\ref{thm:miniLaw}.
To obtain a similar result under the assumptions of Theorem~\ref{thm:derBrw}, we use the tightness of
\[
  M_n - n \kappa'(\theta)  + \frac{3}{2\theta} \log n,
\]
which was proved  in \cite{Mal16} under the assumptions\footnote{By \cite[Equation (B.1)]{Aid}, if $\kappa''(\theta_0) < \infty$ then~\eqref{eqn:ncsBoundary} implies \cite[Equation~(1.4)]{Mal16}.}
of Theorem~\ref{thm:derBrw}. Using the above observations, the asymptotic behaviour of the Laplace transform of $\mathcal{E}_n$ as $n \to \infty$ can be obtained by studying the asymptotic behaviour of $g_\phi$ under assumption~\eqref{eqn:rvNu} as $x \to -\infty$.
\begin{lemma}
\label{lem:estimate}
Let $\nu$ be a probability distribution satisfying~\eqref{eqn:rvNu}. For all $\phi \in \mathcal{T}$, we have
\[
  \lim_{x \to -\infty} \frac{\e^{-\theta x}}{L(-x)} g_\phi(x) = \int \theta \e^{-\theta z}(1 -\e^{-\phi(z)}) \,\dd z.
\]
\end{lemma}

\begin{proof}
Let $\phi \in \mathcal{T}$. We observe that we can rewrite
\[
  g_\phi(x) = -\log \left( 1 -\int 1 - \e^{-\phi(z+x)} \,\nu(\dd z)\right),
\]
and that $y \mapsto 1 - \e^{-\phi(y)}$ is continuous and compactly supported, i.e., an element of $\mathcal{T}$.
 By the dominated convergence theorem, we have that
\[
\lim_{x \to -\infty} \int 1 - \e^{-\phi(z+x)} \,\nu(\dd z) =0,
\]
which implies
\[
  g_\phi(x) \sim \int 1 - \e^{-\phi(z+x)} \,\nu(\dd z)\quad \text{ as $x \to -\infty$.}
\]
 Therefore, it is enough to show that
\begin{equation}
  \label{eqn:lemAim}
  \lim_{x \to -\infty} \int \frac{\e^{-\theta x}}{L(-x)} \psi(x + z) \,\nu (\dd z) = \int \theta \e^{-\theta z} \psi(z) \,\dd z
\end{equation}
 for all $\psi \in \mathcal{T}$ to prove Lemma~\ref{lem:estimate}.
 Using~\eqref{eqn:rvNu} together with the Uniform Convergence Theorem (\citealp[Theorem~1.5.2]{BGT}), we observe that for all $a < b$,
 \begin{align*}
 & \lim_{x \to -\infty} \frac{\e^{-\theta x}}{L(-x)} \nu([a-x,b-x)) \\
  &\qquad= \lim_{x \to -\infty} \left( \e^{-\theta a}\frac{L(a-x)}{L(-x)} \frac{\e^{\theta(a- x)}}{L(a-x)} \nu([a-x,\infty))
  - \e^{-\theta b}\frac{L(b-x)}{L(-x)} \frac{\e^{\theta(b- x)}}{L(b-x)} \nu([b-x,\infty)) \right)
  \\
  &\qquad=\e^{-\theta a} - \e^{-\theta b}=\int_a^b \theta \e^{-\theta z} \,\dd z,
 \end{align*}
from which we deduce~\eqref{eqn:lemAim}  by approaching $\psi$ by a family of
stair functions, thereby completing the proof.
\end{proof}

To simplify the notation, in the rest of the article, we write for all $\phi \in \mathcal{T}$ and $\theta > 0$
\[
  c_\phi(\theta) = \int \theta \e^{-\theta z}(1 -\e^{-\phi(z)}) \,\dd z.
\]
We can then restate Lemma~\ref{lem:estimate} as $g_\phi(x) \sim c_\phi(\theta)L(-x) \e^{\theta x}$ as $x \to -\infty$ under assumption~\eqref{eqn:rvNu}.

Similarly to Lemma~\ref{lem:estimate}, we observe that an upper bound for the right tail of $\nu$ implies a similar upper bound for $g_\phi$ as $x \to -\infty$.
\begin{lemma}
\label{lem:estimateBound}
Let $\nu$ be a probability distribution satisfying~\eqref{eqn:NuBounded}. For all $\phi \in \mathcal{T}$, there exists $C' > 0$ such that for all $x \in \R$,
\[
  g_\phi(x) \leq C' \e^{\theta x}.
\]
\end{lemma}

\begin{proof}
Let $\phi \in \mathcal{T}$, as $\phi$ is compactly supported, there exists $B > 0$ such that $\phi(z) = 0$ for all $z < -B$. Therefore,
\[
  \e^{-g_\phi(x)} = \int \e^{-\phi(x+z)} \,\nu(\dd z) \geq \nu((-\infty,-x-B)) \geq 1 - C \e^{\theta (x+B)}.
\]
As a result, we deduce that
\[
  \limsup_{x \to -\infty} \e^{-\theta x} g_\phi(x) \leq \lim_{x \to-\infty} -\e^{-\theta x} \log \left(1 - C \e^{\theta(x+B)}\right) = C \e^{\theta B}.
\]
Using that $g_\phi$ is bounded, the proof is now complete.
\end{proof}

To complete the proofs of Theorems~\ref{thm:miniBrw},~\ref{thm:miniLaw} and~\ref{thm:derBrw}, it will be enough to show that
\[
  \sum_{|u|=n} g_\phi(S_u - m_n) \text{ converges in probability},
\]
and to identify its limit. This is mainly done using the so-called many-to-one lemma (see~\citealp[Theorem~1.1]{Shi}), that we now state.
\begin{lemma}
\label{lem:manytoone}
Let $\theta > 0$ such that $\kappa(\theta) < \infty$. There exists a random walk $(T_n, n \geq 0)$ such that for all measurable non-negative function $f$, we have
\[
  \E\left[ \sum_{|u|=n} \e^{\theta S_u - n \kappa(\theta)} f(S_{u_1},\ldots,S_{u_n}) \right] = \E\left[ f(T_1,\ldots, T_n) \right].
\]
Moreover, $\E[T_1] = \kappa'(\theta)$ whenever this quantity is well-defined.
\end{lemma}

\section{Proof of the theorems}
\label{sec:proofs}

We prove in this section our main theorems. We first consider the asymptotic behaviour of $\mathcal{E}_n$ below the boundary case, i.e., when $\theta < \theta_0$, assuming that this quantity is well-defined. We then turn to the proof of Theorem~\ref{thm:derBrw}, i.e., assuming that $\theta  =\theta_0$. Finally, we prove Theorem~\ref{thm:aboveBoundary} in Section~\ref{subsec:above}.

\subsection{BRW below the boundary case}

In this section, $\theta$ is a fixed positive constant, and we assume that $\kappa(\theta) < \infty$ and $(W_n(\theta), n \geq 0)$ is uniformly integrable. We denote by $W_\infty(\theta) = \lim_{n \to \infty} W_n(\theta)$ the almost sure limit of this martingale. We start by proving Theorem~\ref{thm:miniBrw}.
\begin{proof}[Proof of Theorem~\ref{thm:miniBrw}]
Let $\phi \in \mathcal{T}$, using Lemma~\ref{lem:estimate}, under assumption~\eqref{eqn:expTail}, for all $\epsilon > 0$, there exists $A > 0$ such that for all $x \leq -A$, we have
\begin{equation}
  \label{eqn:abound}
  \left|g_\phi(x) - c_\phi(\theta) L \e^{\theta x} \right| \leq \epsilon \e^{\theta x}.
\end{equation}
Observe that for all $a, b \geq 0$, we have $|\e^{-a} - \e^{-b}| \leq |a - b|\wedge 1$ using that $\exp$ is $1$-Lipschitz on $\R_-$. Therefore,
\begin{align*}
  \E \left[  \left| \e^{-\crochet{\tau_{-m_n} \mathcal{X}_n ,g_\phi}} - \e^{-c_\phi(\theta)W_n(\theta)}  \right| \right]
  &\leq  \E\left[  \left| \crochet{\tau_{-m_n} \mathcal{X}_n ,g_\phi} - c_\phi(\theta) W_n(\theta)  \right| \wedge 1   \right] \\
  &\leq  \E\left[ \left| \sum_{|u|=n} h_\phi(S_u-m_n) \right| \wedge 1 \right],
\end{align*}
where $h_\phi(x) = g_\phi(x) - c_\phi(\theta) L \e^{\theta x}$, using that $\e^{-\theta m_n} = L^{-1}\e^{-n \kappa(\theta)}$. By~\eqref{eqn:abound}, we have
\[
  \E\left[ \ind{M_n \leq m_n - A}\left| \sum_{|u|=n} h_\phi(S_u-m_n) \right| \wedge 1 \right] \leq \epsilon L^{-1}   \E[W_n(\theta)],
\]
hence
\begin{equation*}
  \E \left[  \left| \e^{-\crochet{\tau_{-m_n}\mathcal{X}_n ,g_\phi}} - \e^{-c_\phi(\theta)W_n(\theta)}  \right| \right]
  \leq \epsilon   L^{-1}  + \P(M_n \geq m_n - A).
\end{equation*}
Letting $n \to \infty$ and then $\epsilon \downarrow 0$, we conclude, by first using Lemma~\ref{lem:asCv} and then using Lebesgue's dominated convergence theorem and~\eqref{eqn:lapTransform}, that 
\begin{align*}
  \lim_{n \to \infty} \E\left[\e^{-\crochet{\tau_{-m_n}\mathcal{E}_n,\phi}} \right] &=  \lim_{n \to \infty} \E\left[\e^{-c_\phi(\theta)\sum_{|u|=n} \e^{\theta S_u -n \kappa(\theta)}} \right]\\
  &=\E\left[ \exp\left( - W_\infty(\theta) \int \theta \e^{-\theta z}(1 - \e^{-\phi(z)}) \,\dd z \right)\right].
\end{align*}
To complete the proof, it is then enough to observe that this limit is the Laplace transform of the PPP($\theta W_\infty(\theta) \e^{-\theta z} \,\dd z$).
\end{proof}

Theorem~\ref{thm:miniLaw} follows from very similar computations. We use the extra integrability condition $\kappa(\theta + \delta) < \infty$ to guarantee that under our stated assumptions, $M_n - m_n$ almost surely decays linearly, and the condition $\kappa(\theta - \delta) < \infty$ to control the contributions of particles far from position $m_n$.
\begin{proof}[Proof of Theorem~\ref{thm:miniLaw}]
We compute the asymptotic behaviour of the Laplace transform of $\tau_{-m_n} \mathcal{E}_n$ using similar computations as for the proof of Theorem~\ref{thm:miniBrw}. We recall that for all $\phi \in \mathcal{T}$, we have
\[
  \E\left[ \exp\left( -\crochet{\tau_{-m_n}\mathcal{E}_n,\phi}\right)\right] = \E\left[ \exp\left(-\sum_{|u|=n} g_\phi(S_u-m_n)\right) \right],
\]
and that $c_1 = \left(\frac{\kappa(\theta)}{\theta} - \kappa'(\theta) \right)^\alpha$. Let $ \eta > 0$ such that $4\eta < \frac{\kappa(\theta)}{\theta}-\kappa'(\theta)$. Using Lemma~\ref{lem:estimate} and the regular variations of $L$ at $\infty$, we deduce  that almost surely, for all $n$ large enough,
\begin{align*}
  &(1 - \tfrac{2}{c_1^{1/\alpha}} \eta)^{|\alpha|}c_1 c_\phi(\theta)\sum_{|u|=n} \e^{\theta S_u-n \kappa(\theta)}\ind{|S_u-n \kappa'(\theta)| \leq n \eta}\\
  &\qquad\leq \sum_{|u|=n} g_\phi(S_u-m_n) \ind{|S_u - n \kappa'(\theta)|\leq n \eta}\\
  &\qquad\qquad\leq (1 - \tfrac{2}{c_1^{1/\alpha}} \eta)^{-|\alpha|}c_1 c_\phi(\theta)\sum_{|u|=n} \e^{\theta S_u-n \kappa(\theta)}\ind{|S_u-n \kappa'(\theta)| \leq n \eta}.
\end{align*}
Here we use the Uniform Convergence Theorem (\citealp[Theorem~1.5.2]{BGT}), which states that $\lim_{z \to \infty} \frac{L(\lambda z)}{L(z)} = \lambda^{\alpha}$ holds uniformly on compact subsets of $(0,\infty)$. Note that there exists $r > 0$ so that
\[
  [(1 - \tfrac{3}{c_1^{1/\alpha}} \eta)^{|\alpha|},(1-\tfrac{3}{c_1^{1/\alpha}}\eta)^{-|\alpha|}] \subset [1 - r\eta,1+r\eta].
\]

Using Lemma~\ref{lem:manytoone}, together with the law of large numbers  for the random walk $(T_n, n\geq0)$, we observe that
\begin{equation}
 \lim_{n\to\infty} \E\left[  \sum_{|u|=n} \e^{\theta S_u-n \kappa(\theta)}\ind{|S_u-n \kappa'(\theta)| > n \eta} \right] = \lim_{n\to\infty}\P\left( \left|T_n - \E[T_n]\right| > n \eta\right) =0.
\label{eqn:errorterm0}
 \end{equation}
Consequently, using the a.s. convergence of $W_n(\theta)$ to $W_\infty(\theta)$, we deduce that
\[
  \lim_{n \to \infty} \sum_{|u|=n} \e^{\theta S_u-n \kappa(\theta)}\ind{|S_u-n \kappa'(\theta)| \leq n \eta} = W_\infty(\theta) \quad \text{ in probability.}
\]
Hence, we conclude that
\[
   \P\left( \left| \sum_{|u|=n} \! g_\phi(S_u-m_n) \ind{|S_u - n \kappa'(\theta)|\leq n \eta}\! - c_1 c_\phi(\theta) W_\infty(\theta)\right| > r \eta c_1 c_\phi(\theta) W_\infty(\theta) \right)
\]
converges to $0$ as $n \to \infty$. In order to show that $\sum_{|u|=n} g_\phi(S_u-m_n)$ converges to $c_1 c_\phi(\theta) W_\infty(\theta)$ in probability and complete the proof of Theorem~\ref{thm:miniLaw}, we now show that
\begin{equation}
\lim_{n \to \infty} \sum_{|u|=n} g_\phi(S_u - m_n) \ind{|S_u-n\kappa'(\theta)| > n\eta}  = 0 \quad \text{ in probability.}
 \label{eqn:errorterm}
 \end{equation}

On the one hand,  using that $\vartheta \mapsto \kappa(\vartheta)$ is $\mathcal{C}^2$ on $(\theta - \delta,\theta + \delta)$, and that $\theta \kappa'(\theta) - \kappa(\theta) < 0$, we observe that $\kappa(\vartheta)/\vartheta$ is decreasing on $[\theta,\theta + \delta]$ for sufficiently small $\delta$,  which implies $\frac{\kappa(\theta)}{\theta}>\frac{\kappa(\theta+\delta)}{\theta+\delta}$. Hence, by Lemma~\ref{lem:asCv}, there exists $\epsilon > 0$ such that
\[
  \lim_{n \to \infty} M_n - n\left(\frac{\kappa(\theta)}{\theta} - \epsilon\right) = -\infty \quad \text{a.s.}
\]
As $\log L(n)/n \to 0$ as $n \to \infty$, we conclude that almost surely, for all $n$ large enough, $M_n - m_n \leq - \epsilon n/2$. Therefore, almost surely, for $n$ large enough
\begin{align*}
  &\sum_{|u|=n} g_\phi(S_u-m_n) \ind{S_u - n \kappa'(\theta) > n \eta}\\
  &\qquad = \sum_{|u|=n} g_\phi(S_u-m_n) \ind{S_u - n \kappa'(\theta) > n \eta, S_u-m_n < - n\epsilon /2}\\
 & \qquad \leq   2c_\phi(\theta) \sum_{|u|=n} L(m_n - S_u) \frac{\e^{\theta S_u - n \kappa(\theta)}}{L(n)} \ind{S_u - n \kappa'(\theta) > n \eta, S_u-m_n < - n\epsilon /2}.
\end{align*}
Therefore, using again the Uniform Convergence Theorem (\citealp[Theorem~1.5.2]{BGT}), there exists a constant $C$ depending on $\epsilon$ and $\eta$ such that almost surely, for all $n$ large enough
\[
  \sum_{|u|=n} g_\phi(S_u-m_n) \ind{S_u - n \kappa'(\theta) > n \eta} \leq C \sum_{|u|=n} \e^{\theta S_u - n \kappa(\theta)} \ind{S_u-n \kappa'(\theta) > n\eta}.
\]
This, together with~\eqref{eqn:errorterm0}, implies that
\begin{equation}
  \label{eqn:above}
  \lim_{n \to \infty} \sum_{|u|=n} g_\phi(S_u-m_n) \ind{S_u - n \kappa'(\theta) > n \eta} = 0 \quad \text{ in probability.}
\end{equation}
With similar computations, we observe that for all $B > 0$, we have
\begin{equation}
  \label{eqn:semibelow}
  \lim_{n \to \infty} \sum_{|u|=n} g_\phi(S_u-m_n)\ind{S_u - n \kappa'(\theta) < - n \eta} \ind{S_u - m_n > - n B} = 0 \quad \text{ in probability.}
\end{equation}

On the other hand, using that $\lim_{x \to \infty} x^{-1}\log L(x) =0$, we observe that for all $x$ large enough, we have
\[
  g_\phi(-x) \leq 2 c_\phi(\theta) L(x) \e^{-\theta x} \leq 4 c_\phi(\theta) \e^{-(\theta - \delta/2)x}.
\]
Therefore, for all $n$ large enough, we have
\begin{align*}
  \sum_{|u|=n} g_\phi(S_u-m_n) \ind{S_u-m_n < -nB}
  &\leq 4 c_\phi(\theta) \sum_{|u|=n} \e^{(\theta - \delta/2) (S_u - m_n)} \ind{S_u - m_n < -nB}\\
  &\leq 4 c_\phi(\theta) W_n(\theta-\delta) \e^{n \kappa(\theta-\delta) - (\theta -\delta) m_n - \delta n B/2}.
\end{align*}
Using that $W_n(\theta - \delta)$ converges almost surely, we observe that for $B$ large enough, we have
\begin{equation}
  \label{eqn:fullbelow}
  \lim_{n \to \infty} \sum_{|u|=n} g_\phi(S_u-m_n) \ind{S_u - m_n< - n B} = 0 \quad \text{ in probability.}
\end{equation}

Consequently, combining~\eqref{eqn:above},~\eqref{eqn:semibelow} and~\eqref{eqn:fullbelow}, we get~\eqref{eqn:errorterm}, which implies
\[
  \lim_{n \to \infty} \sum_{|u|=n} g_\phi(S_u - m_n) = c_1 c_\phi(\theta) W_\infty(\theta) \quad \text{ in probability.}
\]
Using the dominated convergence theorem together with~\eqref{eqn:lapTransform}, we conclude that for all $\phi \in \mathcal{T}$,
\[
  \lim_{n \to \infty} \E\left[\e^{-\crochet{\tau_{-m_n}\mathcal{E}_n,\phi}} \right] = \E\left[ \exp\left( - c_1 W_\infty(\theta)c_\phi(\theta) \right)\right].
\]
We can now complete the proof, observing that the right-hand side is the Laplace transform of the PPP($c_1\theta W_\infty(\theta) \e^{-\theta x} \,\dd x$).
\end{proof}

\subsection{BRW in the boundary case}

We turn in this section to the proof of Theorem~\ref{thm:derBrw}. The proof follows a similar scheme as the one used above, with the added introduction of a shaving procedure. More precisely, for all $A > 0$, we set
\[ 
  \mathcal{G}_n(A) = \{u\in\calU : |u|=n ,\; S_{u_k} \leq k \kappa'(\theta) + A  \quad \forall\,  k \leq n \}, 
\]
the set of particles that stayed at all times below the line $x \mapsto x \kappa'(\theta) + A$. Using that $\lim_{n \to \infty} M_n - n \kappa'(\theta) = -\infty$ a.s. we observe that almost surely, for $A$ large enough, we have $\mathcal{G}_n(A) = \{u\in\calU : |u|=n\}$. In other words, writing
\[
  \mathcal{S}_A := \{\sup_{n \in \N} M_n - n \kappa'(\theta) < A\},
\]
we observe that on the event $\mathcal{S}_A$, we have $\mathcal{G}_n(A) = \{u\in\calU : |u|=n\}$ for all $n \in \N$, and that $\lim_{A\to \infty} \P(\mathcal{S}_A) = 1$.

As a first step towards the proof of Theorem~\ref{thm:derBrw}, we show that no particle in $\mathcal{G}_n(A)$ above position $m_n - \epsilon n^{1/2}$ contributes to the extremal process with high probability, as soon as $\epsilon>0$ is small enough.
\begin{lemma}
\label{lem:lemma}
Under the conditions and notation of Theorem~\ref{thm:derBrw}, for all $B > 0$, we have
\[
  \lim_{\epsilon \downarrow 0} \limsup_{n \to \infty} \P(  \exists u\in\calU : |u|=n , S_u \geq m_n - \epsilon n^{1/2}, S_u + Y_u \geq m_n - B) = 0.
\]
\end{lemma}

\begin{proof}
We first recall from \cite[Theorem 1.1]{Mal16}, that
\[
  \lim_{\epsilon \downarrow 0} \limsup_{n \to \infty} \P(M_n \geq n \kappa'(\theta)- \frac{3}{2\theta} \log n + \epsilon^{-1}) = 0.
\]
Letting $a_n=\frac{1}{\theta}\left(\log n+  \log L\left(\sqrt{n}\right)\right)$, we can rewrite this as
\[
  \lim_{\epsilon \downarrow 0} \limsup_{n \to \infty} \P(M_n \geq m_n-a_n  + \epsilon^{-1}) = 0.
\]
Moreover, we have $\lim_{A\to \infty} \P(\mathcal{S}_A)  = 1$. Therefore it is enough to prove that for all $A > 0$ large enough,
\begin{equation}
  \label{eqn:aimLemma}
  \lim_{\epsilon \downarrow 0} \limsup_{n \to \infty} \P\left(\exists u \in \mathcal{G}_n(A) : \begin{array}{l} m_n - S_u \in [a_n - \epsilon^{-1}, \epsilon n^{1/2}],\\
   S_u + Y_u \geq m_n - B\end{array} \right) = 0.
\end{equation}
Using the Markov inequality, we have
\begin{align*}
  & \P \left(\exists u \in \mathcal{G}_n(A) : m_n - S_u \in [a_n - \epsilon^{-1}, \epsilon n^{1/2}], S_u + Y_u \geq m_n - B \right)\\
  &\qquad\leq  \E\left[ \sum_{u \in \mathcal{G}_n(A)} \ind{m_n - S_u \in [a_n - \epsilon^{-1},\epsilon n^{1/2}]} \nu([m_n - B - S_u, \infty))\right]\\
  &\qquad\leq  2 \E\left[ \sum_{u \in \mathcal{G}_n(A)} \ind{m_n - S_u \in [a_n - \epsilon^{-1},\epsilon n^{1/2}]} L(m_n- S_u) \e^{\theta (S_u + B - m_n)} \right]
\end{align*}
for all $n$ large enough, where we used  the independence between $Y$ and $S$, and the fact that $a_n \to \infty$, therefore we can apply~\eqref{eqn:rvNu} to $m_n - S_u$.

We then use the formula of $m_n$ and the many-to-one lemma to compute
\begin{align*}
  &\E\left[ \sum_{u \in \mathcal{G}_n(A)} \ind{m_n - S_u \in [a_n - \epsilon^{-1},\epsilon n^{1/2}]} L(m_n- S_u) \e^{\theta (S_u - m_n)} \right]\\
 &\qquad=  n^{1/2}\E\left[ \sum_{u \in \mathcal{G}_n(A)} \frac{L(m_n-S_u)}{L(n^{1/2})} \e^{\theta S_u - n \kappa(\theta)} \ind{m_n - S_u \in [a_n -\epsilon^{-1}, \epsilon n^{1/2}]} \right]\\
&\qquad=  n^{1/2}\E\left[ \frac{L(\hat{m}_n-\hat{T}_n)}{L(n^{1/2})} \ind{\hat{m}_n-\hat{T}_n \in [a_n - \epsilon^{-1},\epsilon n^{1/2}], - \hat{T}_j \geq -A, j \leq n} \right],
\end{align*}
where $\hat{T}_k = T_k - k \kappa'(\theta)$ and $\hat{m}_n = m_n - n \kappa'(\theta)$. Let $\rho \in (0,\alpha + 2)$, we define $\mathfrak{L} : x \mapsto x^{\rho - \alpha} L(x)$. Observe that $\mathfrak{L}$ is a regularly varying function at $\infty$ with index $\rho > 0$. Therefore, by the Uniform Convergence Theorem (\citealp[Theorem~1.5.2]{BGT}), for all $\delta>0$, there exists $N_{\delta}$ such that for all $x>N_{\delta}$,
\begin{equation*}
  \left| \frac{L(\lambda x)}{L(x)} -\lambda^{\alpha} \right| = \lambda^{\alpha-\rho}\left| \frac{\mathfrak{L}(\lambda x)}{\mathfrak{L}(x)} -\lambda^{\rho} \right| < \delta\lambda^{\alpha-\rho} \quad\text{ for all $\lambda \in(0,2\epsilon]$.}
\end{equation*}
As a result, for all $n$ large enough, we have
\begin{align*}
  &\E\left[ \frac{L(\hat{m}_n-\hat{T}_n)}{L(n^{1/2})} \ind{\hat{m}_n-\hat{T}_n \in [a_n - \epsilon^{-1},\epsilon n^{1/2}], - \hat{T}_j \geq -A, j \leq n} \right]\\
&\qquad\leq  \E\left[ \left( \frac{\hat{m}_n-\hat{T}_n}{n^{1/2}}\right)^\alpha \ind{\hat{m}_n- \hat{T}_n \in [a_n - \epsilon^{-1},\epsilon n^{1/2}], -\hat{T}_j \geq - A, j \leq n} \right]\\
&\qquad\qquad +  \delta \E\left[ \left( \frac{\hat{m}_n-\hat{T}_n}{n^{1/2}}\right)^{\alpha - \rho} \ind{\hat{m}_n- \hat{T}_n \in [a_n - \epsilon^{-1},\epsilon n^{1/2}], -\hat{T}_j \geq - A, j \leq n} \right].
\end{align*}
We now compute this quantity using the ballot theorem to the centred random walk $(-\hat{T}_n)$ with finite variance.

Using e.g.~\cite[Lemma~4.1]{AiS2}, there exist $C > 0$ and $h > 0$ such that for all $n \in \N$, $a \geq 0$ and $b \geq -a$, we have
\[
  \P( - \hat{T}_n  \in [b,b+h],  -\hat{T}_j \geq - a, j \leq n) \leq C \frac{\left((a+1) \wedge n^{1/2} \right)\left((a+b+1) \wedge n^{1/2}\right)}{n^{3/2}}.
\]
Let $\gamma \in (-2,0)$. For all $n$ large enough, we have
\begin{align*}
  &\E\left[ \left( \hat{m}_n-\hat{T}_n \right)^\gamma \ind{\hat{m}_n- \hat{T}_n \in [a_n - \epsilon^{-1},\epsilon n^{1/2}], -\hat{T}_j \geq - A, j \leq n} \right]\\
&\qquad\leq \sum_{k=1}^{\ceil{\epsilon n^{1/2}/h}} (kh)^\gamma \P(\hat{m}_n-\hat{T}_n \in [kh, (k+1)h], - \hat{T}_j \geq -A, j \leq n)\\
&\qquad\leq \frac{C h^\gamma (A+1)}{n^{3/2}} \sum_{k = 1}^{\ceil{\epsilon n^{1/2}/h}} k^\gamma (A + kh + C'\log n+1)\\
&\qquad\leq \frac{2C h^\gamma (A+1)^2}{n^{3/2}} \sum_{k = 1}^{\ceil{\epsilon n^{1/2}/h}} k^\gamma (kh + C'\log n),
\end{align*}
where $C'$ is chosen so that $|\hat{m}_n| \leq C' \log n$ for all $n$ large enough.
Decomposing this sum for $k \leq \log n$ and $k \geq \log n$, we obtain, for $n$ large enough
\begin{align*}
  &\E\left[ (\hat{m}_n - \hat{T}_n )^{\gamma} \ind{\hat{m}_n-\hat{T}_n \in [a_n - \epsilon^{-1},\epsilon n^{1/2}],  -\hat{T}_j \geq -A, j \leq n} \right]\\
&\qquad\leq  \frac{2C h^\gamma (A+1)^2(C'+h)}{n^{3/2}} \left( (\log n )^{\gamma + 2} + \sum_{k = \ceil{\log n}}^{\ceil{\epsilon n^{1/2}/h}} k^{\gamma +1} \right)\\
&\qquad\leq  \frac{K_\gamma}{n^{3/2}} (\log n )^{\gamma + 2} + \frac{K'_\gamma}{n^{3/2}} \epsilon^{\gamma + 2} n^{\gamma/2 +1},
\end{align*}
for some $K_\gamma,K'_\gamma>0$ since $\gamma  +1 > -1$. As a result, letting $n \to \infty$, we obtain
\begin{align*}
  &\limsup_{n \to \infty}\P\left(\exists u \in \mathcal{G}_n(A) : \begin{array}{l} m_n - S_u \in [a_n - \epsilon^{-1}, \epsilon n^{1/2}],\\
     S_u + Y_u \geq m_n - B\end{array} \right)\\
&\qquad\leq  2\e^{\theta B}\limsup_{n \to \infty} n^{1/2} \E\left[ \frac{L(\hat{m}_n-\hat{T}_n)}{L(n^{1/2})} \ind{\hat{m}_n - \hat{T}_n \in [a_n - \epsilon^{-1},\epsilon n^{1/2}], -\hat{T}_j \geq - A, j \leq n} \right]\\
&\qquad\leq  2\e^{\theta B} K'_{\alpha}  \epsilon^{\alpha + 2} + 2\e^{\theta B} \delta K'_{\alpha-\rho}  \epsilon^{\alpha-\rho+2}.
\end{align*}
Using that $\alpha > \alpha - \rho > -2$, we conclude that~\eqref{eqn:aimLemma} holds, which completes the proof.
\end{proof}

We now turn to the proof of Theorem~\ref{thm:derBrw}.

\begin{proof}[Proof of Theorem~\ref{thm:derBrw}]
For any $n \in \N$ and $\epsilon > 0$, we write
\[
  \tilde{\mathcal{E}}_n^{(\epsilon)} = \sum_{|u|=n} \delta_{S_u + Y_u} \ind{S_u \leq  m_n - \epsilon n^{1/2}}.
\]
Using Lemma~\ref{lem:lemma}, together with the inequality $|\e^{-a} - \e^{-b}| \leq |a-b| \wedge 1$, we observe that for all~$\phi \in \mathcal{T}$,
\begin{align}
  \label{eqn:lemma}
  &\limsup_{\epsilon \downarrow 0} \limsup_{n \to \infty} \left|  \E\left[\e^{-\crochet{\tau_{-m_n}\mathcal{E}_n,\phi}}\right] - \E\left[\e^{-\crochet{\tau_{-m_n}\tilde{\mathcal{E}}_n^{(\epsilon)},\phi}}\right] \right|\\
&\qquad\leq  \lim_{\epsilon \downarrow 0} \limsup_{n \to \infty}  \E \left[\left( \sum_{|u|=n} \phi(S_u+Y_u-m_n)  \ind{S_u \geq  m_n - \epsilon n^{1/2}}    \right)\wedge 1\right]
   = 0.
\end{align}
It is therefore enough to study the asymptotic behaviour of $\E\left[\e^{-\crochet{\tau_{-m_n}\tilde{\mathcal{E}}_n^{(\epsilon)},\phi}}\right]$ to identify the limiting distribution of $\tau_{-m_n} \mathcal{E}_n$.

Let $\phi \in \mathcal{T}$. Using the same computations as in~\eqref{eqn:lapTransform}, we have
\begin{align*}
  \E\left[ \exp\left( - \crochet{  \tau_{-m_n} \tilde{\mathcal{E}}^{(\epsilon)}_n,\phi}\right)\right]
  = \E\left[ \exp\left( - \sum_{|u|=n} g_\phi(S_u-m_n) \ind{S_u \leq m_n - \epsilon n^{1/2}} \right) \right].
\end{align*}
We study the asymptotic behaviour of $\sum_{|u|=n} g_\phi(S_u-m_n) \ind{S_u \leq m_n - \epsilon n^{1/2}}$ as $n \to \infty$ then $\epsilon \downarrow 0$. By Lemma~\ref{lem:estimate}, we get that for all $\delta > 0$, for all  large enough $n$,
\begin{align*}
  &(1 -\delta) c_\phi(\theta)  n^{1/2}  \sum_{|u|=n} \frac{L(m_n-S_u)}{L(n^{1/2})} \e^{\theta S_u - n \kappa(\theta)}\ind{ m_n-S_u \geq \epsilon n^{1/2}}\\
  &\qquad\leq \sum_{|u|=n} g_\phi(S_u-m_n) \ind{ m_n-S_u \geq \epsilon n^{1/2}}\\
  &\qquad\qquad\leq (1 +\delta) c_\phi(\theta)  n^{1/2}  \sum_{|u|=n} \frac{L(m_n-S_u)}{L(n^{1/2})} \e^{\theta S_u - n \kappa(\theta)}\ind{ m_n-S_u \geq \epsilon n^{1/2}}.
\end{align*}

Recall that $L$ is regularly varying with index $\alpha < 0$. We use again the Uniform Convergence Theorem (\citealp[Theorem~1.5.2]{BGT}), yielding
\[
  \lim_{x \to \infty} \sup_{\lambda > \epsilon/2} \left| \frac{L(\lambda x)}{L(x)} - \lambda^\alpha \right| = 0.
\]
As a result, for all $0 < \delta < \epsilon$, for all $n$ large enough, we have
\begin{align*}
   &(1 -\delta) c_\phi(\theta)  n^{1/2}  \sum_{|u|=n} \underline{h}_{\epsilon,\delta}((m_n-S_u)/n^{1/2}) \e^{\theta S_u - n \kappa(\theta)}\\
   &\qquad\leq \sum_{|u|=n} g_\phi(S_u-m_n) \ind{ m_n-S_u \geq \epsilon n^{1/2}}\\
   &\qquad\qquad\leq (1 +\delta) c_\phi(\theta)  n^{1/2}  \sum_{|u|=n} \bar{h}_{\epsilon,\delta}((m_n-S_u)/n^{1/2}) \e^{\theta S_u - n \kappa(\theta)},
\end{align*}
where $\underline{h}_{\epsilon,\delta}$ and $\bar{h}_{\epsilon,\delta}$ are continuous functions such that, for all $x \in \R$,
\begin{align*}
    &(x^\alpha - \delta) \ind{[\epsilon+\delta,\infty)} \leq \underline{h}_{\epsilon,\delta}(x) \leq (x^\alpha - \delta) \ind{[\epsilon,\infty)}  \\
    \text{ and }& (x^\alpha + \delta) \ind{[\epsilon,\infty)} \leq \bar{h}_{\epsilon,\delta}(x) \leq (x^\alpha + \delta) \ind{[\epsilon - \delta,\infty)}.
\end{align*}

As a result, using a combination of~\cite[Theorem~1.2]{Mad16} and~\cite[Theorem~1.1]{AiS}, then letting $\delta\downarrow 0$ we conclude that
\begin{align*}
& \lim_{n \to \infty} \sum_{|u|=n} g_\phi(S_u-m_n) \ind{ m_n-S_u \geq \epsilon n^{1/2}}\\
&\quad = \sqrt{\frac{2}{\pi \kappa''(\theta)}} c_\phi(\theta) Z_\infty \E\left[ \left(R_1 \sqrt{\kappa''(\theta)} \right)^\alpha \ind{R_1 \sqrt{\kappa''(\theta)}  \geq \epsilon} \right] \quad\text{ in probability,}
\end{align*}
where $(R_t, t\in[0,1])$ is a Brownian meander. Since $R_1$ has  Rayleigh distribution, we observe that $\E\left[ \left(R_1\right)^\alpha \right] < \infty$. Now, letting $\epsilon \downarrow 0$, we obtain by monotonicity that
\begin{align*}
  &\lim_{\epsilon \downarrow 0} \lim_{n \to \infty}  \sum_{|u|=n } g_\phi(S_u-m_n) \ind{m_n - S_u \geq \epsilon n^{1/2}}\\
  &\qquad= \sqrt{\frac{2}{\pi \kappa''(\theta)}} c_\phi(\theta) Z_\infty \E\left[ \left(R_1 \sqrt{\kappa''(\theta)} \right)^\alpha \right] \quad\text{ in probability.}
\end{align*}

Finally, we set
\[
  c_2 = \sqrt{\frac{2}{\pi \kappa''(\theta)}} \E\left[ \left(R_1 \sqrt{\kappa''(\theta)} \right)^\alpha \right]  =\sqrt{\frac{2}{\pi \kappa''(\theta)}}\left( 2 \kappa''(\theta) \right)^{\frac{\alpha}{2}} \Gamma\left(\frac{\alpha}{2}+1\right).
\]
We now conclude, by equation~\eqref{eqn:lemma}, that
\begin{align*}
 \E\left[\e^{-c_2 c_\phi(\theta) Z_\infty}\right] = \lim_{\epsilon \downarrow 0} \lim_{n \to \infty}   \E\left[\e^{-\crochet{\tau_{-m_n}\tilde{\mathcal{E}}^{(\epsilon)}_n,\phi}}\right]
 = \lim_{n \to \infty} \E\left[\e^{-\crochet{\tau_{-m_n}\mathcal{E}_n,\phi}}\right].
\end{align*}
Identifying the Laplace transform of the  PPP($c_2  \theta  Z_\infty \e^{-\theta x} \,\dd x$), the proof of Theorem~\ref{thm:derBrw} is now complete.
\end{proof}

\subsection{BRW above the boundary case}
\label{subsec:above}

We prove in this section Theorem~\ref{thm:aboveBoundary} as a consequence of the convergence of the extremal process of the BRW observed by \cite{Mad17}. Recall that by~\cite[Theorem~1.1]{Mad17}, for all $\phi \in \mathcal{T}$, we have
\[
  \lim_{n \to \infty} \E\left[ \e^{- \crochet{\mathcal{Z}_n,\phi}} \right] = \E\left[ \e^{-\crochet{\mathcal{Z}_\infty,\phi}} \right],
\]
where $\mathcal{Z}_\infty$ is the limiting extremal process of the BRW, a decorated, randomly shifted Poisson point process with exponential intensity. Moreover,~\cite[Proposition 2.1]{Mad17} shows that $\sum_{|u|=n} \e^{\theta (S_u - m_n)}$ is tight for all~$\theta > \theta_0$.
Using these two results and assuming~\eqref{eqn:NuBounded}, we show that for all $\phi \in \mathcal{T}$,
 \[
  \lim_{n \to \infty} \E\left[ \e^{- \crochet{\mathcal{Z}_n,g_\phi}} \right] = \E\left[ \e^{-\crochet{\mathcal{Z}_\infty,g_\phi}} \right],
\]
which implies the convergence of the extremal process of the last progeny modified BRW.

\begin{proof}[Proof of Theorem~\ref{thm:aboveBoundary}]
Let $K > 0$, we write $\chi_K$ a function in $\mathcal{T}$ such that
\[
  \indset{[-K,K]} \leq \chi_K \leq \indset{[-2K,2K]}.
\]
Let $\phi \in \mathcal{T}$, using~\eqref{eqn:lapTransform}, we can write for all $n \in \N$,
\begin{align*}
& \left| \E\left[ \e^{-\crochet{\tau_{-m_n}\mathcal{E}_n,\phi}}  \right] - \E\left[ \e^{-\crochet{\mathcal{Z}_\infty,g_\phi}} \right] \right|\\
&\qquad =  \left|  \E\left[ \e^{ - \sum_{|u|=n} g_\phi(S_u - m_n) }\right] - \E\left[ \e^{-\crochet{\mathcal{Z}_\infty,g_\phi}} \right]  \right|
 \leq I_1(n,K) + I_2(n,K) + I_3(K),
\end{align*}
where we have set
\begin{align*}
  I_1(n,K) &= \left| \E\left[ \e^{ - \sum_{|u|=n} g_\phi(S_u - m_n) } \right] - \E\left[ \e^{ - \sum_{|u|=n} \chi_K g_\phi(S_u - m_n) } \right] \right|\\
  I_2(n,K) &= \left| \E\left[ \e^{ - \sum_{|u|=n} \chi_K g_\phi(S_u - m_n) } \right] - \E\left[ \e^{-\crochet{\mathcal{Z}_\infty,\chi_K g_\phi}} \right] \right|\\
  I_3(K) &= \left| \E\left[ \e^{-\crochet{\mathcal{Z}_\infty,\chi_K g_\phi}} \right] - \E\left[ \e^{-\crochet{\mathcal{Z}_\infty,g_\phi}} \right] \right|.
\end{align*}

Observe that $\chi_K g_\phi \in \mathcal{T}$. Therefore, by~\cite[Theorem~1.1]{Mad17}, we know that for all $K>0$, $\lim_{n \to \infty} I_2(n,K)= 0$. Additionally, $\lim_{K \to \infty} I_3(K) = 0$ by Lebesgue's dominated convergence theorem. To complete the proof of Theorem~\ref{thm:aboveBoundary}, we now bound $I_1(n,K)$ in $n$.

Let $\theta_1\in(\theta_0,\theta)$, using the inequality $|\e^{-a} - \e^{-b}| \leq |a-b| \wedge 1$ for all $a,b \geq 0$, we have
\begin{align*}
  I_1(n,K) & \leq \E\left[ \left(\sum_{|u|=n} (1-\chi_K) g_\phi(S_u - m_n)\right) \wedge 1 \right]\\
  &\leq \E\left[ \left(\sum_{|u|=n}  g_\phi(S_u - m_n)\ind{|S_u-m_n| \geq K} \right) \wedge 1 \right]\\
  &\leq \E\left[ \left(\sum_{|u|=n}  g_\phi(S_u - m_n)\ind{S_u - m_n \leq -K} \right) \wedge 1 \right]\\
  &\qquad \qquad + \E\left[ \left(\sum_{|u|=n}  g_\phi(S_u - m_n)\ind{S_u-m_n \geq K} \right) \wedge 1 \right],
\end{align*}
 by the sub-additivity $(a+b)\wedge1 \leq (a\wedge1)+(b\wedge1)$ for all $a,b \geq 0$. We first observe that
\[
  \E\left[ \left(\sum_{|u|=n}  g_\phi(S_u - m_n)\ind{S_u-m_n \geq K} \right) \wedge 1 \right]  \leq \P(M_n-m_n  \geq  K),
\]
which converges to $0$ uniformly in $n$ as $K \to \infty$, by tightness of $(M_n -m_n)$. Moreover, using Lemma~\ref{lem:estimateBound}, we have
\begin{align*}
  &\E\left[ \left(\sum_{|u|=n}  g_\phi(S_u - m_n)\ind{S_u-m_n \leq - K} \right) \wedge 1 \right]\\
   &\qquad \leq  \E\left[ \left( C'\sum_{|u|=n} \e^{\theta (S_u - m_n)} \ind{S_u-m_n \leq -K} \right) \wedge 1 \right]\\
   &\qquad \leq C' \E\left[ \ind{\sum_{|u|=n} \e^{\theta_1(S_u-m_n)} \leq K} \sum_{|u|=n} \e^{\theta (S_u - m_n)} \ind{S_u-m_n \leq -K} \right]  +\P\left(\sum_{|u|=n} \e^{\theta_1(S_u-m_n)} \geq K\right),
\end{align*}
for an arbitrary $\theta_1 \in (\theta_0,\theta)$. As a result,
\begin{align*}
\E\left[ \left(\sum_{|u|=n}  g_\phi(S_u - m_n)\ind{S_u-m_n \leq - K} \right) \wedge 1 \right]
\leq C' K \e^{-K(\theta-\theta_1)} + \P\left(\sum_{|u|=n} \e^{\theta_1(S_u-m_n)} \geq K\right).
\end{align*}
Using~\cite[Proposition 2.1]{Mad17}, we observe that $\sum_{|u|=n} \e^{\theta_1(S_u-m_n)}$  is  tight, and therefore we conclude that
\[
  \lim_{K\to \infty} \sup_{n \in \N} I_1(n,K) = 0.
\]

Finally,
\begin{align*}
  \lim_{n \to \infty} \E\left[ \e^{-\crochet{\tau_{-m_n}\mathcal{E}_n,\phi}}  \right]& =\E\left[ \e^{-\crochet{\mathcal{Z}_\infty,g_\phi}} \right]
  = \E\left[ \exp\left( - \sum_{i \in \N} \phi(z_i + Y_i) \right)\right],
\end{align*}
which completes the proof.
\end{proof}


\section*{Acknowledgements}
The authors would like to thank the anonymous referees for their careful reading and detailed comments, which have helped to improve the paper.

\end{document}